\documentclass[12pt]{amsart}
\usepackage{amscd,amsmath,amssymb,amsfonts}
\usepackage{dsfont,enumerate}
\usepackage[cmtip, all]{xy}

\usepackage{graphicx,amssymb,amsmath,amsfonts,amsthm,amscd,hyperref,textcomp,relsize,colortbl}
\usepackage{tikz-cd}
\usetikzlibrary{babel}

\theoremstyle{plain}
\newtheorem{thm}{Theorem}
\newtheorem{lem}[thm]{Lemma}

\newtheorem{prop}[thm]{Proposition}

\theoremstyle{definition}

\newtheorem{rmk}[thm]{Remark}

\numberwithin{thm}{section}
\numberwithin{equation}{thm}

\newcommand{\rank}{{\rm rank}}

\newcommand{\Trace}{{\rm Trace}}

\newcommand{\Norm}{{\rm Norm}}

\newcommand{\sF}{{\mathcal F}}
\newcommand{\sG}{{\mathcal G}}
\newcommand{\sH}{{\mathcal H}}

\newcommand{\sL}{{\mathcal L}}

\newcommand{\A}{{\mathbb A}}

\newcommand{\C}{{\mathbb C}}

\newcommand{\F}{{\mathbb F}}
\newcommand{\G}{{\mathbb G}}

\renewcommand{\P}{{\mathbb P}}
\newcommand{\Q}{{\mathbb Q}}

\newcommand{\Z}{{\mathbb Z}}
\newcommand{\bfZ}{{\mathbf Z}}

\newcommand{\triv}{{\mathds{1}}}
\newcommand{\ord}{\mathrm {ord}}

\newcommand{\CC}{\mathbb C}

\newcommand{\Aut}{\mathrm{Aut}}

\newcommand{\GL}{\mathrm{GL}}
\newcommand{\SL}{\mathrm{SL}}
\newcommand{\SO}{\mathrm{SO}}
\newcommand{\PGL}{\mathrm{PGL}}
\newcommand{\PSL}{\mathrm{PSL}}
\newcommand{\Co}{\mathrm{Co}}
\newcommand{\Alt}{\mathsf {A}}
\newcommand{\SSS}{\mathsf {S}}

\newcommand{\Mellin}{{\mathsf {Mellin}}}
\newcommand{\Teich}{{\mathsf {Teich}}}
\newcommand{\Swan}{{\mathsf {Swan}}}

\begin{document}
\title[Rigid local systems with monodromy group $\Co_3$]{Rigid local systems with monodromy group the Conway group $\Co_3$}
\author{Nicholas M. Katz, Antonio Rojas-Le\'{o}n, and Pham Huu Tiep}
\address{Department of Mathematics, Princeton University, Princeton, NJ 08544, USA}
\email{nmk@math.princeton.edu}
\address{Departamento de \'{A}lgebra, Universidad de Sevilla, c/Tarfia s/n, 41012 Sevilla, Spain}
\email{arojas@us.es}
\address{Department of Mathematics, Rutgers University, Piscataway, NJ 08854, USA}
\email{tiep@math.rutgers.edu}
\thanks{The second author was partially supported by MTM2016-75027-P (Ministerio de Econom\'ia y Competitividad) and FEDER. The third author gratefully acknowledges the support of the NSF (grant 
DMS-1840702).} 

\maketitle

\begin{abstract} 
We first develop some basic facts about certain sorts of rigid local systems on the affine line in characteristic $p >0$. We then apply them to
exhibit a number of rigid local systems of rank $23$ on the affine line in characteristic $p=3$ whose arithmetic and geometric monodromy groups are
the Conway group $\Co_3$ in its orthogonal irreducible representation of degree $23$. 
\end{abstract}

\tableofcontents

\section*{Introduction}
In the first section, we recall the general set up, and some basic results. In the second section, we generalize the criteria of \cite[Thm. 1]{R-L} and \cite[5.1]{Ka-RLSA} for finite (arithmetic and geometric) monodromy to more general local systems. In the third section, we apply these criteria to show that
certain local systems have finite (arithmetic and geometric) monodromy groups. In the fourth section, we show that the finite monodromy groups in question are the Conway group $\Co_3$ in its $23$-dimensional irreducible orthogonal representation.

\section{The basic set up, and general results}
We fix a prime number $p$, a prime number $\ell \neq p$, and a nontrivial $\overline{\Q_\ell}^\times$-valued additive character $\psi$ of $\F_p$. For $k/\F_p$ a finite extension, we denote by  $\psi_k$ the nontrivial additive character of $k$ given by $\psi_k :=\psi\circ \Trace_{k/\F_p}$. In perhaps more down to earth terms, we fix a nontrivial $\Q(\mu_p)^\times$-valued additive character $\psi$, of $\F_p$, and a field embedding of
$\Q(\mu_p)$ into  $\overline{\Q_\ell}$ for some $\ell \neq p$.

Given an integer $D \ge 3$ which is prime to $p$, a finite extension $k/\F_p$, and a polynomial $f(x) \in k[x]$ of degree $D$, we form the local system $\sF_{p,D,f,\triv}$ on $\A^1/k$ whose trace function is given as follows: for $L/k$ a finite extension, at $L$-valued points $t \in \A^1(L) = L$, is given by
$$t \mapsto -\sum_{x \in L}\psi_L(f(x) + tx).$$

This is a geometrically irreducible rigid local system, being the Fourier Transform of the rank one local system $\sL_{\psi(f(x))}$. It has rank $D-1$, it is totally wild at $\infty$, and each of its $D-1$ $I(\infty)$-slopes is $D/(D-1)$. It is pure of weight one.  [When $f(x)$ is $x^D$, it is the local system  $\sF(\F_p,\psi, \triv, D)$ of \cite{Ka-RLSA}.]

Suppose in addition we are given a {\bf nontrivial} character $\chi$ of $k^\times$. For $L/k$ a finite extension, we obtain a nontrivial character $\chi_{L}$ of $L^\times$ by defining $\chi_{L} :=\chi \circ \Norm_{L/k}$.
We then form the local system $\sF_{p,D,f,\chi}$ on $\A^1/k$ whose trace function is given as follows: for $L/k$ a finite extension, at $L$-valued points $t \in \A^1(L) = L$, is given by
$$t \mapsto -\sum_{x \in L}\psi_L(f(x) + tx)\chi_{L}(x).$$
This too is a geometrically irreducible rigid local system, being the Fourier Transform of the rank one local system $\sL_{\psi(f(x))}\otimes \sL_{\chi(x)}$. It has rank $D$. Its $I(\infty)$ representation is the direct sum of the tame character $\sL_{\overline{\chi}(x)}$ with a totally wild representation of rank $D-1$, each of whose $D-1$ $I(\infty)$-slopes is $D/(D-1)$ \cite[Thm. 2.4.3]{Lau}. It is pure of weight one.  [When $f(x)$ is $x^D$, it is the local system  $\sF(\F_p,\psi, \chi, D)$ of \cite{Ka-RLSA}.]

\begin{lem}{\rm (Primitivity Lemma)}\label{primitivity} We have the following results.
\begin{enumerate}[\rm(i)]
\item If both $D$ and $D-1$ are prime to $p$, the local system $\sF_{p,D,f,\triv}$ is not geometrically induced, i.e., there is no 
triple $(U,\pi, \sH)$ consisting of a
connected smooth curve $U/\overline{k}$, a finite etale map $f:U \rightarrow \A^1/\overline{k}$ of degree $d \ge 2$, and a local system $\sH$ on $U$ such that there exists an isomorphism of $\pi_\star \sH$ with (the pullback to $\A^1/\overline{k}$ of) $\sF_{p,D,f,\triv}$.
\item When $D$ is prime to $p$, then for any nontrivial $\chi$,  the local system $\sF_{p,D,f,\chi}$ is not geometrically induced.
\end{enumerate}
\end{lem}

\begin{proof} In case (i), our local system is an Airy sheaf (the Fourier transform of a nonconstant lisse sheaf on $\A^1$ of rank one). 
By a result \cite[11.1]{Such} if an Airy sheaf is induced, it is Artin-Schreier induced, so has rank divisible by $p$.

For (ii), we argue as follows. If such a triple exists, then we have an equality of Euler characteristics
$$EP(U,\sH) =EP(\A^1/\overline{k},\pi_\star \sH) = EP(\A^1/\overline{k},\sF_{p,D,f,\chi}) .$$
Denote by $X$ the complete nonsingular model of $U$, and by $g_X$ its genus. Then $\pi$ extends to a finite flat map of $X$ to $\P^1$, and the Euler-Poincar\'{e} formula gives
$$EP(U,\sH) = \rank(\sH)(2-2g_X -\#(\pi^{-1}(\infty) )) \sum_{w \in \pi^{-1}(\infty)}\Swan_w(\sH),$$
$$EP(\A^1/\overline{k},\sF_{p,D,f,\chi}) = \rank(\sF_{p,D,f,\chi}) -\Swan_{\infty}(\sF_{p,D,f,\chi}) = D-D =0.$$
So we have the equality
$$0 = \rank(\sH)(2-2g_X -\#(\pi^{-1}(\infty) )) \sum_{w \in \pi^{-1}(\infty)}\Swan_w(\sH).$$

We first bound the genus $g_X$. We must have $g_X \le 0$, otherwise the factor $2-2g_X -\#(\pi^{-1}(\infty))$ is $\le -1$, and the right 
hand side is strictly negative. 

Thus $g_X=0$,  and we have
$$0 =  \rank(\sH)(2-\#(\pi^{-1}(\infty) ))- \sum_{w \in \pi^{-1}(\infty)}\Swan_w(\sH).$$

If $\#(\pi^{-1}(\infty)) =1$, then $U$ is $\P^1 \setminus ({\rm one \ point}) \cong \A^1$, and so $\pi$ is a finite etale map of $\A^1$ to itself of degree $> 1$. But any such map has degree divisible by $p$. Indeed, when the map is given by the polynomial $F(x)$, the hypothesis is
that for every $t \in \overline{k}$, the two equations $F(x)=t, F'(x)=0$ have no common solution. If $F'$ had a zero, say $a$, then
$a$ would be a solution of $F(x)=F(a), F'(x)=0$. Thus $F'$ has no zeroes, so is some nonzero constant $A$, and hence $F(x)$ is of the form $G(x)^p +Ax$.

We cannot have $\#(\pi^{-1}(\infty))  \ge 3$, otherwise the factor $2-\#(\pi^{-1}(\infty))$ is strictly negative, and the right side is then strictly negative.
It remains to treat the case when  $\#(\pi^{-1}(\infty)) =2$ (and $g_X=0$). Throwing the two points to $0$ and $\infty$, we have a finite etale
map
$$\pi:\G_m \rightarrow \A^1.$$
The equality of EP's now gives
$$0 =\Swan_0(\sH) + \Swan_\infty(\sH).$$
Thus $\sH$ is lisse on $\G_m$ and everywhere tame, so a successive extension of lisse, everywhere tame sheaves of rank one.
But $\pi_\star \sH$ is irreducible, so $\sH$ must itself be irreducible, hence of rank one, and either $\overline{\Q_\ell}$ or an $\sL_{\rho}$.
[It cannot be $\overline{\Q_\ell}$, because $\pi_\star \overline{\Q_\ell}$ is not irreducible when $\pi$ has degree $>1$; by adjunction $\pi_\star \overline{\Q_\ell}$ contains $\overline{\Q_\ell}$.]
Now consider the maps induced by $\pi$ on punctured formal neighborhoods
$$\pi(0): \G_m{(0)} {\rightarrow  \A^1}{(\infty)},\ \ \pi(\infty): \G_m{(\infty)} \rightarrow  {\A^1}{(\infty)}.$$
The $I(\infty)$-representation of $\sF_{p,D,f,\chi}$ is then the direct sum
$$\pi(0)_\star \sL_{\rho} \oplus \pi(\infty)_\star \sL_{\rho}.$$
Denote by $d_0$ and $d_\infty$ their degrees. For any tame $\sL_\Lambda$, we have
$$\pi(0)^\star \sL_\Lambda \cong  \sL_{\Lambda^{d_0}}, \ \ \pi(\infty)^\star \sL_\Lambda \cong  \sL_{\Lambda^{d_\infty}}.$$
Since the tame character group is divisible, there exist $\Lambda_0$ with $\Lambda_0^{d_0}=\rho$ (in fact, as many as the prime to $p$ part $n_0$ of $d_0=n_0\times({\rm a\ power \ of\ }p)$), and there exist $\Lambda_\infty$ with $\Lambda_\infty^{d_\infty}=\rho$ (in fact, as many as the prime to $p$ part $n_\infty$ of $d_\infty=n_\infty \times ({\rm a\ power \ of\ }p)$,

 Thus if $\sF_{p,D,f,\chi}$ were induced, its $I(\infty)$ representation would contain at least two tame characters. 
\end{proof}

Let $k$ be a finite extension of $\F_p$, $\ell \neq p$, $U/k$ a smooth, geometrically connected $k$-scheme of relative dimension $\ge 0$, and $\sG$ a $\overline{\Q_\ell}$ local system on $U$ of rank $d \ge 1$. Viewing $\sG$ as a representation of $\pi_1(U)$, say
$$\rho_{\sG}:\pi_1(U) \rightarrow \GL_d(\overline{\Q_\ell}),$$
we get its arithmetic monodromy group $G_{arith}$, defined to be the Zariski closure of the image of $\pi_1(U)$. Inside $\pi_1(U)$ we
its normal subgroup $\pi_1^{geom}(U):=\pi_1(U\otimes_k \overline{k})$. The Zariski closure of the image of $\pi_1^{geom}(U)$ is the
geometric monodromy group $G_{geom}$. Thus we have
$$G_{geom} \lhd G_{arith} \subset \GL_d(\overline{\Q_\ell}).$$

When we apply this general machine to the local system  $\sF_{p,D,f,\triv}$ on $\A^1/k$, we get its 
$$G_{geom} \lhd G_{arith} \subset \GL_{D-1}(\overline{\Q_\ell}).$$

Similarly, for any nontrivial $\chi$, when we apply the general machine to the local system $\sF_{p,D,f,\chi}$ on $\A^1/k$, we get its 
$$G_{geom} \lhd G_{arith} \subset \GL_{D}(\overline{\Q_\ell}).$$

\begin{lem}\label{p-group} {\rm ($p$-subgroups of $G_{geom}$)} Suppose both $D \ge 3$ and $D-1$ are prime to $p$. Denote by $f$ the multiplicative order of $p$ in $(\Z/(D-1)\Z)^\times$, so that $\F_{p^f}$ is the extension $\F_p(\mu_{D-1})$ of $\F_p$ obtained by adjoining the $D-1$ roots of unity. Then for the particular local systems 
$\sF_{p,D,x^D,\triv}$ or $\sF_{p,D,x^D,\chi}$ with any nontrivial $\chi$, the image in $G_{geom}$ of the wild inertia group $P(\infty)$ is isomorphic to (the additive group of)  $\F_{p^f}$.
\end{lem}
\begin{proof}In all cases, the local system, restricted to $\G_m$, descends through the $D$'th power map. For $\sF_{p,D,x^D,\triv}$,  the descent is given explicitly in terms of trace functions as
$$ t \mapsto -\sum_{x \in k}\psi_k(x^D/t + x).$$
For $\sF_{p,D,x^D,\chi}$, we must first choose a character $\Lambda$ with $\Lambda^D =\overline{\chi}$. Then the descent is given explicitly interms of trace function as
$$ t \mapsto -\sum_{x \in L^\times}\Lambda_L(t)\psi_k(x^D/t + x)\chi_L(x).$$
In fancier terms, the first descent is to a Kloosterman sheaf of rank $D-1$, the second is to a hypergeometric sheaf of type $(D,1)$, cf. \cite [2.1]{Ka-RLSA}.

In all cases, the wild part $W$ of the $I(\infty)$ representation has rank $D-1$ and all slopes $1/(D-1)$. Because $D-1 \ge 2$, one knows \cite[8.6.3]{Ka-ESDE} that $W$ is unique up to tensoring with a tame character and performing a multiplicative translation. Thus the underlying $P(\infty)$-representation is unique up to a multiplicative translation, which does not change its image in $G_{geom}$. Because $D-1$ is prime to $p$, we obtain one such $W$ by forming the direct image by $D-1$ power
$$W := [D-1]_\star \sL_{\psi(x)}.$$

Because $D-1$ is prime to $p$, the image of $P(\infty)$ does not change if we pass to the pullback 
$$[D-1]^\star W =[D-1]^\star [D-1]_\star  \sL_{\psi(x)}   \cong \bigoplus_{\zeta \in \mu_{D-1}} \sL_{\psi(\zeta x)}.$$
In other words, the image of $P(\infty)$ is the abelian group whose character group consists of all monomials
$$\otimes _{\zeta \in \mu_{D-1}} \sL_{\psi(\zeta x)}^{\otimes n_{\zeta}}= \sL_{\psi(\sum_{\zeta \in \mu_{D-1}} n_{\zeta} \zeta x)}.$$
as each  $n_{\zeta}$ runs over $\Z/p\Z$. This character group is thus the subring $\F_p[\mu_{D-1}] \subset \F_{p^f}$ consisting of all $\F_p$-linear combinations of elements of $ \mu_{D-1}$. But this subring, being a finite integral domain, is itself a field. It lies in $\F_{p^f}$, and
contains $\mu_{D-1}$, so it is  $\F_{p^f}$. One knows that (by the trace),  $\F_{p^f}$ is its own Pontrayagin dual.
\end{proof}

\section{Criteria for finite monodromy}
We first recall the basic underlying result, cf \cite[8.14]{Ka-ESDE}. [It is stated there for a local system on an open curve, but the "on a curve" hypothesis never enters the proof. Also, the exact rule of the hypothesis of geometric irreducibility is less clear than it might be, cf. Remark \ref{nonirred} below.

Let $k$ be a finite extension of $\F_p$, $\ell \neq p$, $U/k$ a smooth, geometrically connected $k$-scheme of relative dimension $\ge 0$, and $\sG$ a $\overline{\Q_\ell}$ local system on $U$ of rank $d \ge 1$. We have its geometric and arithmetic monodromy groups
$$G_{geom} \lhd G_{arith} \subset \GL_d(\overline{\Q_\ell}).$$

\begin{prop}\label{finmonocrit} Suppose we have $(k,\ell,U, \sG)$ as above. Suppose further that $\sG$ is pure of weight zero for all embeddings of $\overline{\Q_\ell}$ into $\C$. Consider the following four conditions.
\begin{enumerate}[\rm(a)]
\item $G_{arith}$ is finite.
\item All traces of $\sG$ are algebraic integers. More precisely, for every finite extension $L/k$, and for every point $u \in U(L)$, $\Trace(Frob_{L,u} |\sG)$ is an algebraic integer.
\item $G_{geom}$ is finite.
\item $\det(\sG)$ is arithmetically of finite order.
\end{enumerate}
Then we have the implications
$${\rm (a)} \implies {\rm (b)} \implies {\rm (c)},\ {\rm (b)} \implies {\rm (d)},$$
and if $\sG$ is geometrically irreducible, we have ${\rm (a)} \iff {\rm (b)} \iff {\rm (c)}$.
\end{prop}
\begin{proof}The implications ${\rm (a)} \implies {\rm (b)}$ and ${\rm (a)} \implies {\rm (c)}$ are both obvious. 

We next show that ${\rm (b)} \implies {\rm (c)}$ and ${\rm (b)} \implies {\rm (d)}$. If (b) holds, then all the eigenvalues of
each Frobenius are algebraic integers which, by purity, have absolute value $1$ at all archimedean places, hence are roots of
unity. Because $\sG$ is realizable over some finite extension $E_\lambda$ of ${\Q_\ell}$, each of these roots of unity
lies in a extension of $E_\lambda$ of degree at most the rank of $\sG$. As there are only finitely many such extensions inside $\overline{\Q_\ell}$,
all eigenvalues are roots of unity in a fixed finite extension of $\Q_\ell$, so are all $N$'th roots of unity for some $N$. 
Applying this same argument to the rank one local system $\det(\sG)$, we see that $\det(\sG)^{\otimes N}$ is trivial, i.e. we see that
${\rm (b)} \implies {\rm (d)}$.
By Chebotarev and Zariski density, every $\gamma \in G_{arith}$ has $\gamma^N$ unipotent. 
In particular, every element in $G_{geom}$, and hence every element in the identity component $G_{geom}^0$ has $N$'th power unipotent. By Deligne \cite[1.3.8 and 3.4.1 (iii)]{De-Weil II}, $G_{geom}^0$ is a semisimple algebraic group.
Looking at elements of a maximal torus, we see that $G_{geom}^0$ has rank $0$, hence $G_{geom}^0$ is trivial and thus $G_{geom}$ is finite.

When $\sG$ is geometrically irreducible, the implication ${\rm (c)} \implies {\rm (a)}$, using (d), is proven in \cite[8.14.3.1]{Ka-ESDE}.
\end{proof}

\begin{rmk}\label{nonirred}Here is an example to show that geometric irreducibility is needed to prove that ${\rm (b)} \implies {\rm (a)}$. Take any $U/k$, and take on it the pullback from $Spec(k)$ of the geometrically constant local system $\beta^{deg}$ for $\beta$ the upper unipotent matrix $\small \left(\begin{array}{cc}1 & 1 \\ 0 &1 \end{array}\right)$.
Then all Frobenius eigenvalues are $1$, $G_{geom}$ is trivial, but $G_{arith}$ is the upper unipotent subgroup $\small\left(\begin{array}{cc}1 & \star \\ 0 &1 \end{array}\right)$ of $SL(2)$.

On the other hand, suppose (b) holds. If we pass from $\sG$ viewed as a representation of $\pi_1^{arith}$, to its semisimplification $\sG^{ss}$, (which has the same trace function as $\sG$), then
$G_{arith,\sG^{ss}}$ is reductive. Then the fact that every element in this group has $N$'power unipotent shows that the identity
component $G_{arith,\sG^{ss}}^0$ is trivial (look at a maximal torus), and hence $G_{arith,\sG^{ss}}$ is finite.
\end{rmk}

We now define the sort of multi-parameter local systems it will be convenient to work with. We fix an integer $D \ge 3$ prime to p,
and a sequence of integers of length $r  \ge 1$,
$$1= d_1 < d_2 <\ldots<d_r <D,$$
each of which is itself prime to $p$. [Because of Artin-Schreier reduction, requiring the $d_i$ to be prime to $p$ is no loss of generality.]
We form the local system $\sF(p, D, d_1,\ldots,d_r,\triv)$ on $\A^{r-1}$ whose trace function is as follows: for $K/\F_p$ a finite extension, and  $(t_,\ldots,t_r) \in \A^r(K)$, the trace function is
$$(t_,\ldots,t_r) \mapsto -\sum_{x \in K}\psi_K(x^D + \sum_i t_i x^{d_i}).$$
When $p$ is odd, we also form the local system $\sF(p, D, d_1,\ldots,d_r,\chi_2)$ on $\A^{r-1}$ whose trace function is as follows: for $K/\F_p$ a finite extension, and  $(t_,\ldots,t_r) \in \A^r(K)$, the trace function is
$$(t_,\ldots,t_r) \mapsto -\sum_{x \in K^\times}\psi_K(x^D + \sum_i t_i x^{d_i})\chi_{2,K}(x).$$

Because $d_1 =1$, these local systems are geometrically irreducible (indeed, for $r=1$ these are $\sF(p,D,x^D,\triv)$ and $\sF(p,D,x^D,\chi_2)$). If $r \ge 2$,
their pullbacks to $\A^1$ by freezing $t_2=\ldots=t_r=0$ are $\sF(p,D,x^D,\triv)$ and $\sF(p,D,x^D,\chi_2)$. They are pure of weight one,
for all embeddings of $\overline{\Q_\ell}$ into $\C$. Moreover, their traces lie in $\Z[\zeta_p]$.

For each of them, we now fix a version of a half Tate twist as follows: we choose an algebraic integer $\alpha$ which is (some root of unity)$\sqrt{p}$, typically some fourth root of $p^2$.


Using our chosen $\alpha$, we then define
$$\sG(p, D, d_1,\ldots,d_r,\triv):= \sF(p, D, d_1,\ldots,d_r,\triv)\otimes (1/\alpha)^{deg},$$
and, when $p$ is odd, 
$$\sG(p, D, d_1,\ldots,d_r,\chi_2):= \sF(p, D, d_1,\ldots,d_r,\chi_2)\otimes (1/\alpha)^{deg}.$$
 
The local systems $\sG_{p,D,f,\triv}$ and, when $p$ is odd, $\sG_{p,D,f,\chi_2}$, are pure of weight zero. Their determinants are arithmetically of finite order, cf. the proof of \cite[Lemma 1.1]{Ka-RL}.

For later use, we record the following facts about autoduality.
\begin{lem}\label{autoduality}Suppose that $D$ and all the $d_i$ are {\bf odd}. Then there is a preferred choice of $\alpha$, as follows.
\begin{itemize}
\item[(1)]For $\alpha$ either choice of $\sqrt{p}$, $\sG(p, D, d_1,\ldots,d_r,\triv)$ has
$$G_{geom} \lhd G_{arith} \subset {\rm Sp}_{D-1}(\overline{\Q_\ell}).$$
\item[(2)]If $p$ is odd, write $D=2\delta +1$. Then for $\alpha :=-\chi_2((-1)^\delta D)g(\psi,\chi_2)$  ($g(\psi,\chi_2)$ the quadratic Gauss sum over $\F_p$), $\sG(p, D, d_1,\ldots,d_r,\chi_2)$ has
$$G_{geom} \lhd G_{arith} \subset {\rm SO}_D(\overline{\Q_\ell}).$$
\item[(3)] Denote by $\Q(\zeta_p)^{+}$ the real subfield of $\Q(\zeta_p)^{+}$. In case (1), the traces attained lie in  $\Q(\zeta_p)^{+}(\sqrt{p})$. In case (2), they lie in $\Q(\zeta_p)^{+}$.
\end{itemize}
\end{lem}
\begin{proof}Assertion (1) is  Poincar\'{e} duality. Assertion (2) is proved in \cite[1.7]{Ka-NG2}. Assertion (3) is obvious from applying complex conjugation to the formulas for the traces.
\end{proof}


The local system $\sG(p, D, d_1,\ldots,d_r,\triv)$ on $\A^r/\F_p$ has its
$$G_{geom}\lhd G_{arith} \subset \GL_{D-1}(\overline{\Q_\ell}).$$

Similarly, when $p$ is odd, the local system $\sG(p, D, d_1,\ldots,d_r,\chi_2)$ 
on $\A^r/\F_p$ has its
$$G_{geom}\lhd G_{arith} \subset \GL_D(\overline{\Q_\ell}).$$

We now apply  \ref{finmonocrit} to these local systems.
\begin{prop}\label{arithfin}The following conditions are equivalent.
\begin{enumerate}
\item [\rm (a)]$\sG(p, D, d_1,\ldots,d_r,\triv)$ has finite $G_{geom}$.
\item [\rm (a-bis)] $\sF(p, D, d_1,\ldots,d_r,\triv)$ has finite $G_{geom}$.
\item [\rm (b)] $\sG(p, D, d_1,\ldots,d_r,\triv)$ has finite $G_{arith}$.
\item [\rm (c)] All traces of  $\sG(p, D, d_1,\ldots,d_r,\triv)$ are algebraic integers.
\end{enumerate}
When $p$ is odd, we have these same equivalences for $\sG(p, D, d_1,\ldots,d_r,\chi_2)$.
\end{prop}

We now perform two successive reduction steps. The first is to pull back these local systems from $\A^r$ to $\G_m^r$, i.e., requiring the
coefficients $t_i$ to all be invertible.
\begin{lem}The following conditions are equivalent.
\begin{enumerate}[\rm(a)]
\item $\sG(p, D, d_1,\ldots,d_r,\triv)$ has finite $G_{arith}$.
\item The pullback of $\sG(p, D, d_1,\ldots,d_r,\triv)$ to $\G_m^r$ has finite $G_{arith}$.
\end{enumerate}
If $p$ is invertible, the same equivalence holds for $\sG(p, D, d_1,\ldots,d_r,\chi_2)$.
\end{lem}
\begin{proof}
That (a) implies (b) is obvious. Because $\G_m^r$ is a dense open set of $\A^r$, $\pi_1(\G_m^r)$ maps onto $\pi_1(\A^r)$, hence (b) implies (a).
\end{proof}

We now form local systems $\sG_{big}(p, D, d_1,\ldots,d_r,\triv)$ and, if $p$ is odd, $\sG_{big}(p, D, d_1,\ldots,d_r,\chi_2)$ on $\G_m^{r+1}$, by letting the coefficient of $x^D$ also vary over invertible scalars. Thus the trace function of $\sG_{big}(p, D, d_1,\ldots,d_r,\triv)$ is 
$$(t_1,\ldots,t_{r+1}) \in \G_m(K)^{r+1} \mapsto \sum_{x \in K}\psi_K(t_{r+1}x^D + \sum_i t_i x^{d_i})/\alpha^{\deg(K/\F_p)}.$$
When $p$ is odd,  the trace function of $\sG_{big}(p, D, d_1,\ldots,d_r,\chi_2)$ is
$$(t_1,\ldots,t_{r+1}) \in \G_m(K)^{r+1} \mapsto \sum_{x \in K^\times}\psi_K(t_{r+1}x^D + \sum_i t_i x^{d_i})\chi_{2,K}(x)/\alpha^{\deg(K/\F_p)}.$$

\begin{lem}\label{big}The following conditions are equivalent.
\begin{enumerate}[\rm(a)]
\item $\sG_{big}(p, D, d_1,\ldots,d_r,\triv)$ on $\G_m^{r+1}$ has finite $G_{arith}$.
\item The pullback of $\sG(p, D, d_1,\ldots,d_r,\triv)$ to $\G_m^r$ has finite $G_{arith}$.
\end{enumerate}
If $p$ is invertible, the same equivalence holds with $\triv$ replaced by $\chi_2$.
\end{lem}

\begin{proof} In both cases, it is obvious that (a) implies (b), since the second local system is the pullback of the first to the locus $t_{r+1}=1$.
To show that (b) implies (a), we argue as follows. Over a finite extension $K$ of $\F_p$, if we make the substitution $x \mapsto \lambda x$
with $\lambda \in K^\times$, the sum
$$ \sum_{x \in K}\psi_K(x^D + \sum_i t_i x^{d_i})/\alpha^{\deg(K/\F_p)}$$
is equal to the sum
$$ \sum_{x \in K}\psi_K(\lambda^d x^D + \sum_i \lambda^{d_i}t_i x^{d_i})/\alpha^{\deg(K/\F_p)}.$$
After the change of variable $t_i \mapsto t_i/\lambda^{d_i}$, this sum is 
$$ \sum_{x \in K}\psi_K(\lambda^d x^D +\sum_i t_i x^{d_i})/\alpha^{\deg(K/\F_p)},$$
still an algebraic integer. In other words, the pullback of $\sG_{big}(p, D, d_1,\ldots,d_r,\triv)$ on $\G_m^{r+1}$ to $\G_m^{r+1}$ by
the finite etale galois map
$$(t_1,\ldots,t_r, t_{r+1}) \mapsto (t_1,\ldots,t_r, t_{r+1}^D)$$
has all its traces algebraic integers, hence has finite $G_{arith}$. But under this finite etale map, the map of $\pi_1$'s makes the source
a subgroup of index $D$ in the target. Thus the $G_{arith}$ for $\sG_{big}(p, D, d_1,\ldots,d_r,\triv)$ contains a finite group as a subgroup of finite index, so is itself finite.

When $p$ is odd, and $\chi =\chi_2$, apply the identical argument. In this case, the $x \mapsto \lambda x$ substitution moves the sum by a factor $\chi_2(\lambda) = \pm 1$, so does not change the fact that the sum is an algebraic integer.
\end{proof}

The sums 
$$ \sum_{x \in K}\psi_K(t_{r+1}x^D + \sum_i t_i x^{d_i})$$
and, when $p$ is odd,  the sums
$$ \sum_{x \in K^\times}\psi_K(t_{r+1}x^D + \sum_i t_i x^{d_i})\chi_{2,K}(x),$$
lie in $\Z[\zeta_p]$. The quantity $\alpha$ in all cases has $\alpha^4=p^2$. The field $\Q(\zeta_p)$ has a unique place over $p$. So these sums will {\bf remain} algebraic integers when divided by $\alpha^{\deg(K/\F_p)}$ if and only if the divided sums be $p$-integral. Equivalently , whenever $K$ is $\F_q$, and $\ord_q$ is the $p$-adic ord, normalized to have $\ord_q(q)=1$, we must have
$$\ord_q\left( \sum_{x \in K =\F_q}\psi_K(t_{r+1}x^D + \sum_i t_i x^{d_i})\right) \ge 1/2,$$
and, when $p$ is odd, we must have
$$\ord_q\left( \sum_{x \in K^\times =\F_q^\times}\psi_K(t_{r+1}x^D + \sum_i t_i x^{d_i})\chi_{2,K}(x)\right) \ge 1/2,$$
for every finite extension $K/\F_p$ and every $r+1$ tuple $(t_1,\ldots,t_{r+1}) \in (K^\times)^{r+1}$.

We now give a generalization of  \cite[Theorem 1]{R-L} to these local systems. In the formulation, we make reference to the $\ord_q$ of various Gauss sums over variable $\F_q$. We view these sums as taking values in $\Q_p(\zeta_p)^{nr}$, the maximal unramified extension of
$\Q_p(\zeta_p)$ (i.e., we adjoin to $\Q_p(\zeta_p)$ all roots of unity of order prime to $p$). This field has a unique $p$-adic $\ord$.

\begin{thm} \label{mellintrick} We have the following results, in which we write $d_{r+1}:=D$.
\begin{enumerate}[\rm(i)]
\item $\sG_{big}(p, D, d_1,\ldots,d_r,\triv)$ has finite $G_{arith}$ if and only if the following condtion holds. 

For every finite extension $K=\F_q$ of $\F_p$, 
and for every $r+1$ tuple of (possibly trivial) multiplicative characters $(\rho_1,\ldots,\rho_{r+1})$ of $K^\times$, not all of which are trivial, such that $\prod_i \rho_i^{d_i} =\triv$, we have
$$\ord_{q}(\prod_i g(\psi_K,\rho_i)) \ge 1/2.$$
\item If $p$ is odd, then $\sG_{big}(p, D, d_1,\ldots,d_r,\chi_2)$ has finite $G_{arith}$ if and only if the following condition holds.

For every finite extension $K=\F_q$ of $\F_p$, 
and for every $r+1$ tuple of (possibly trivial) multiplicative characters $(\rho_1,\ldots,\rho_{r+1})$ of $K^\times$ such that $\prod_i \rho_i ^{d_i}=\chi_{2,K}$, we have
$$\ord_{q}(\prod_i g(\psi_K,\rho_i)) \ge 1/2.$$

\end{enumerate}
\end{thm}

\begin{proof}We first explain the underlying idea. For a fixed $K=\F_q$, we have a function on $(K^\times)^{r+1}$, say
$$F(t_1,\ldots,t_{r+1}),$$
whose values lie in $\Z[\zeta_p]$. We wish to show that each divided value $F(t_1,\ldots,t_{r+1})/\alpha^{\deg(K/\F_p)}$ remains
an algebraic integer, or equivalently that 
$$\ord_q(F(t_1,\ldots,t_{r+1})) \ge 1/2.$$
 For this, we consider the Mellin transform: for each $r+1$ tuple of multiplicative characters $(\rho_1,\ldots,\rho_{r+1})$ of $K^\times$, we look at the the sum
$$\Mellin_F(\rho_1,\ldots,\rho_{r+1}) :=\sum_{(t_1,\ldots,t_{r+1}) \in (K^\times)^{r+1}}F(t_1,\ldots,t_{r+1})\prod_i \rho_i(t_i).$$
We can recover $F$ from $\Mellin_F$ by usual Mellin inversion, which involves division by $(q-1)^{r+1}$, a quantity prime to $p$.
So it suffices to show that each value $\Mellin_F(\rho_1,\ldots,\rho_{r+1})$ has $\ord_q \ge 1/2.$

We first treat assertion (i). The function $F(t_1,\ldots,t_{r+1})$ at hand is 
$$F(t_1,\ldots,t_{r+1}) := \sum_{x \in K}\psi_K(t_{r+1}x^D + \sum_{i=1}^r t_i x^{d_i})=$$
$$=\sum_{x \in K}\psi_K( \sum_{i=1}^{r+1} t_i x^{d_i}) = 1 + F^\times(t_1,\ldots,t_{r+1}),$$
with
$$F^\times(t_1,\ldots,t_{r+1}):=\sum_{x \in K^\times}\psi_K( \sum_{i=1}^{r+1} t_i x^{d_i}).$$

When all the $\rho_i$ are trivial, we have
$$\Mellin_F(\triv,\ldots,\triv) = (q-1)^{r+1} + \Mellin_{F^\times}(\triv,\ldots,\triv),$$
and
$$ \Mellin_{F^\times}(\triv,\ldots,\triv)=\sum_{(t_1,\ldots,t_{r+1}) \in (K^\times)^{r+1}}F^\times(t_1,\ldots,t_{r+1})=$$
$$=\sum_{x \in K^\times}\sum_{(t_1,\ldots,t_{r+1}) \in (K^\times)^{r+1}}\psi_K( \sum_{i=1}^{r+1} t_i x^{d_i})=$$
$$=\sum_{x \in K^\times}\prod_i (\sum_{t_i \in K^\times}\psi_K( t_i x^{d_i})).$$
Each of the $r+1$ summands inside the product is equal to $-1$, because $ x^{d_i}$ is nonzero, so $t_i \mapsto \psi_K( t_i x^{d_i})$ is 
a nontrivial additive character of $K$, and we sum over the nonzero elements. So we find that
$$ \Mellin_{F^\times}(\triv,\ldots,\triv) = (q-1)(-1)^{r+1},$$
and hence
$$\Mellin_F(\triv,\ldots,\triv) = (q-1)^{r+1} + (q-1)(-1)^{r+1},$$
which is divisible by $q$.

When not all the $\rho_i$ are trivial, the constant term of $F$ dies, and we have
$$\Mellin_F(\rho_1\ldots,\rho_{r+1}) =\Mellin_{F^\times}(\rho_1\ldots,\rho_{r+1})=$$
$$=\sum_{x \in K^\times}\prod_i \left(\sum_{t_i \in K^\times}\psi_K( t_i x^{d_i})\rho_i(t_i)\right).$$
Here each of the $r+1$ summands inside the product is easily expressed in terms of Gauss sums:
$$\sum_{t_i \in K^\times}\psi_K( t_i x^{d_i})\rho_i(t_i)= \overline{\rho_i}(x^{d_i})g(\psi_K,\rho_i).$$
So we get
$$\Mellin_F(\rho_1\ldots,\rho_{r+1}) =(\prod_i g(\psi_K,\rho_i))\sum_{x \in K^\times}(\prod_i \overline{\rho_i^{d_i}})(x).$$

The sum over $x \in K^\times$ vanishes unless $\prod_i \rho_i^{d_i} =\triv$. If $\prod_i \rho_i ^{d_i}=\triv$, then we get $(q-1)(\prod_i g(\psi_K,\rho_i))$.

The proof of case (ii) is analogous. Here $F$ is already $F^\times$, and the final formula is
$$\Mellin_F(\rho_1\ldots,\rho_{r+1}) =(\prod_i g(\psi_K,\rho_i))\sum_{x \in K^\times}(\prod_i \overline{\rho_i^{d_i}})(x)\chi_{2,K}(x).$$
\end{proof}

We now reformulate the previous Theorem \ref{mellintrick} in terms of Kubert's $V$ function
$$V: (\Q/\Z)_{prime \ to \ p} \rightarrow [0,1).$$
For $\F_{p^f}$ a finite extension of $\F_p$, and $x \in (\Q/\Z)_{prime \ to \ p}$ with $(p^f-1)x \in \Z$, we have
$$V(x):= \ord_{p^f}(g(\psi_{\F_{p^f}},\Teich^{-x(p^f-1)})),$$
for $$\Teich_{p^f}:\F_{p^f}^\times \cong \mu_{p^f -1}(\Q_p^{nr})$$
the Teichmuller character, characterized by the requirement that for any $x \in \F_{p^f}^\times$, $\Teich_{p^f}(x)$ lifts $x$.
For such an $x$, we have the Stickelberger formula
$$V(x) = (1/f)\sum_{i {\rm (mod }\ f)}<p^ix>.$$

It will also be convenient to introduce a slight variant of Kubert's $V$ function, $V_{RL}$, defined by
$$V_{RL}(x) = V(x) {\rm \ for \ } x \neq 0,\ V_{RL}(0) =1.$$ 
The advantage of this is that the property of the $V$ function
$$V(x) + V(-x) = 1 {\rm \ if \ }x\neq 0$$
becomes the formula
$$V(x) + V_{RL}(-x) =1, {\rm \ for  \  all \ } x.$$

Thus we may reformulate Theorem \ref{mellintrick} as follows, where we ``solve" for $x_1$ in terms of $(x_2,\ldots,x_{r+1})$.

\begin{thm}\label{Vmellin}We have the following results.
\begin{enumerate}[\rm(i)]
\item $\sG_{big}(p, D, d_1,\ldots,d_r,\triv)$ has finite $G_{arith}$ if and only if the following condtion holds. For every list
$(x_2,\ldots,x_{r+1})$ of elements of $(\Q/\Z)_{prime\  to \ p}$ which are not all $0$,
we have the inequality
$$\sum_{i \ge 2}V(x_i) + 1/2 \ge V_{RL}(\sum_{i \ge 2} d_ix_i).$$
\item If p is odd, then $\sG_{big}(p, D, d_1,\ldots,d_r,\chi_2)$ has finite $G_{arith}$ if and only if the following condtion holds. 
For every list of elements
$(x_2,\ldots,x_{r+1})$ of elements of $(\Q/\Z)_{prime \ to \ p}$,
we have the inequality
$$\sum_{i \ge 2}V(x_i) + 1/2 \ge V_{RL}(1/2 + \sum_{i \ge 2} d_ix_i).$$
\end{enumerate}

\end{thm}

We now recall the explicit ``sum of digits" recipe for $V$ and for $V_{RL}$, cf. \cite[Appendix]{Ka-RL}. For an integer $y$, and a power $p^f$ of $p$, we define
$$[y]_{p,f,-}$$ 
to be the sum of the $p$-adic digits of the representative of $y \mod p^f-1$ in $[0,p^f -2]$, and we define
$$[y]_{p,f}$$ 
to be the sum of the $p$-adic digits of the representative of $y \mod p^f-1$ in $[1,p^f -1]$.
Then we have
$$V\left(\frac{y}{p^f-1}\right) = \frac{1}{f(p-1)}[y]_{p,f,-},$$
$$V_{RL}\left(\frac{y}{p^f-1}\right) = \frac{1}{f(p-1)}[y]_{p,f}.$$

With this notation, Theorem \ref{Vmellin} can be restated as

\begin{thm}\label{Vmellin2}We have the following results.
\begin{enumerate}[\rm(i)]
\item $\sG_{big}(p, D, d_1,\ldots,d_r,\triv)$ has finite $G_{arith}$ if and only if the following condtion holds. For every positive integer $f$ and every $r$-tuple of integers
$0\leq x_2,\ldots,x_{r+1} < p^f-1$ which are not all $0$,
we have the inequality
\begin{equation}\label{V2eq1}\left[\sum_{i=2}^{r+1}d_ix_i\right]_{p,f}\leq\sum_{i=2}^{r+1}[x_i]_{p,f,-} + \frac{f(p-1)}{2}\end{equation}
\item If p is odd, then $\sG_{big}(p, D, d_1,\ldots,d_r,\chi_2)$ has finite $G_{arith}$ if and only if the following condtion holds. 
For every positive integer $f$ and every $r$-tuple of integers
$0\leq x_2,\ldots,x_{r+1} < p^f-1$,
we have the inequality
\begin{equation}\label{V2eq2}\left[\sum_{i=2}^{r+1}d_ix_i+\frac{p^f-1}{2}\right]_{p,f}\leq\sum_{i=2}^{r+1}[x_i]_{p,f,-} + \frac{f(p-1)}{2}.\end{equation}
\end{enumerate}

\end{thm}

We also have one further criterion that involves the simpler function $[x]_p:=$ sum of the $p$-adic digits of $x$. We first prove the following

\begin{lem}[Hasse-Davenport relation]\label{HD}
 Let $f,k$ be positive integers and $x\in{\mathbb Z}$. Then we have
 $$
 \left[\frac{p^{fk}-1}{p^f-1}x\right]_{p,fk}=k\cdot[x]_{p,f}
 $$
 and
 $$
 \left[\frac{p^{fk}-1}{p^f-1}x\right]_{p,fk,-}=k\cdot[x]_{p,f,-}
 $$
\end{lem}

\begin{proof}
 If $x\equiv y$ (mod $p^f-1$) then $\frac{p^{fk}-1}{p^f-1}x\equiv\frac{p^{fk}-1}{p^f-1}y$ (mod $p^{fk}-1$), so it suffices to prove it for $0\leq x < p^f-1$ (so $\frac{p^{fk}-1}{p^f-1}x<p^{fk}-1$). But then the result is clear since the $p$-adic expansion of $\frac{p^{fk}-1}{p^f-1}x$ is the concatenation of $k$ copies of the $p$-adic expansion of $x$ (filled with leading 0's so that it has exactly $f$ digits). 
\end{proof}

For use below, we recall the following result from \cite[Prop. 2.2]{Ka-RL}, whose inequalities are used in the proof of Theorem \ref{Vmellin3}.

\begin{prop}\label{inequality}
For strictly positive integers $x$ and $y$, and any $f \ge 1$, we have:
\begin{enumerate}[\rm(i)]
 \item $[x+y]_{p}\leq [x]_{p}+[y]_{p}$;
 \item $[x]_{p,f} \le [x]_p$;
 \item $[px]_{p}=[x]_{p}$.
\end{enumerate}
\end{prop}

\begin{thm}\label{Vmellin3}We have the following results.
\begin{enumerate}[\rm(i)]
\item $\sG_{big}(p, D, d_1,\ldots,d_r,\triv)$ has finite $G_{arith}$ if and only if there exists some real $A \geq 0$ such that for every positive integer $f$ and every $r$-tuple of integers
$0\leq x_2,\ldots,x_{r+1} < p^f-1$ which are not all $0$,
we have the inequality
\begin{equation}\label{V3eq1}\left[\sum_{i=2}^{r+1}d_ix_i\right]_p\leq\sum_{i=2}^{r+1}[x_i]_p + \frac{f(p-1)}{2}+A.\end{equation}
\item If p is odd, then $\sG_{big}(p, D, d_1,\ldots,d_r,\chi_2)$ has finite $G_{arith}$ if and only if there exists some real $A\geq 0$ such that for every positive integer $f$ and every $r$-tuple of integers
$0\leq x_2,\ldots,x_{r+1} < p^f-1$,
we have the inequality
\begin{equation}\label{V3eq2}\left[\sum_{i=2}^{r+1}d_ix_i+\frac{p^f-1}{2}\right]_{p}\leq\sum_{i=2}^{r+1}[x_i]_{p} + \frac{f(p-1)}{2}+A.\end{equation}
\end{enumerate}

\end{thm}

\begin{proof}
 We will prove that the hypotheses of this theorem are equivalent to those of Theorem \ref{Vmellin2}. 
 
 Suppose that there exists $A\geq 0$ such that (\ref{V3eq1}) holds for every $f \ge 1$ and every $ 0 \le x_2,\ldots,x_{r+1} < p^f-1$ which are not all zero. Then $\sum_{i=2}^{r+1}d_ix_i  >0$, and 
$$\begin{aligned}
 \left[\sum_{i=2}^{r+1}d_ix_i\right]_{p,f}& \leq\left[\sum_{i=2}^{r+1}d_ix_i\right]_p \leq\sum_{i=2}^{r+1}[x_i]_p + \frac{f(p-1)}{2}+A\\
 & =\sum_{i=2}^{r+1}[x_i]_{p,f,-} + \frac{f(p-1)}{2}+A.
\end{aligned}$$ 
 In particular, for every positive integer $k$,
 $$
 \left[\sum_{i=2}^{r+1}d_i\frac{p^{fk}-1}{p^f-1}x_i\right]_{p,fk}\leq\sum_{i=2}^{r+1}\left[\frac{p^{fk}-1}{p^f-1}x_i\right]_{p,fk,-} + \frac{fk(p-1)}{2}+A.
 $$
 By Lemma \ref{HD}, dividing by $k$ we get
 $$
 \left[\sum_{i=2}^{r+1}d_ix_i\right]_{p,f}\leq\sum_{i=2}^{r+1}[x_i]_{p,f,-} + \frac{f(p-1)}{2}+A/k.
 $$
 and taking $k\to\infty$ gives us (\ref{V2eq1}). The proof for case (ii) is similar.
 
 Conversely, suppose that (\ref{V2eq1}) holds for every  $f \ge 1$ and every $ 0 \le x_2,\ldots,x_{r+1} < p^f-1$. Let $l$ be an integer such that $\sum_i d_i < p^l$. Then, if $ 0 \le x_2,\ldots,x_{r+1}<p^f-1$, $\sum_{i}d_ix_i< p^{f+l}-1$, so
$$\begin{aligned}
 \left[\sum_{i=2}^{r+1}d_ix_i\right]_p=\left[\sum_{i=2}^{r+1}d_ix_i\right]_{p,f+l} & \leq\sum_{i=2}^{r+1}[x_i]_{p,f+l,-} + \frac{(f+l)(p-1)}{2}\\
 & =\sum_{i=2}^{r+1}[x_i]_{p} + \frac{f(p-1)}{2} + \frac{l(p-1)}{2},\end{aligned}
$$
and (\ref{V3eq1}) holds with $A=l(p-1)/2$.
 
Finally, suppose that (\ref{V2eq2}) holds for every $f \ge 1$ and every $ 0 \le x_2,\ldots,x_{r+1} < p^f-1$. Let $l$ be an integer such that $2\sum_i d_i < p^l$. Then $\sum_i d_ix_i+(p^{f+l}-1)/2<p^{f+l}-1$, so
 $$\begin{aligned}
 1+\left[\sum_{i=2}^{r+1}d_ix_i+\frac{p^f-1}{2}\right]_p & =
 \left[p^{f+l}+\sum_{i=2}^{r+1}d_ix_i+\frac{p^f-1}{2}\right]_p\\
& =\left[p^f\frac{p^{l}+1}{2}+\sum_{i=2}^{r+1}d_ix_i+\frac{p^{f+l}-1}{2}\right]_p \\
& \leq\left[\frac{p^l+1}{2}\right]_p+\left[\sum_{i=2}^{r+1}d_ix_i+\frac{p^{f+l}-1}{2}\right]_p\\
& = 1+\frac{l(p-1)}{2}+\left[\sum_{i=2}^{r+1}d_ix_i+\frac{p^{f+l}-1}{2}\right]_{p,f+l}\\
& \leq 1+\frac{l(p-1)}{2}+\sum_{i=2}^{r+1}[x_i]_{p,f+l,-}+\frac{(f+l)(p-1)}{2}\\
& =1+\sum_{i=2}^{r+1}[x_i]_{p}+\frac{f(p-1)}{2}+l(p-1), \end{aligned}$$
 and (\ref{V3eq2}) holds with $A=l(p-1)$.
\end{proof}

\section{Theorems of finite monodromy}
From known results of Kubert, explained in  \cite[4.1,4.2,4.3]{Ka-RLSA}, and the result  \cite[Thm. 3.1]{G-K-T}, we know that $G_{geom}$ and $G_{arith}$ for $\sF_{p,D,x^{D},\chi_2}\otimes \alpha^{-\deg}$ are finite when $q$ is a power of an odd prime $p$ and $D$ is any of
$$\frac{q+1}{2},\;  \frac{q^n+1}{q+1} {\rm \ with \ odd \ }n,\; 2q-1.$$
We will refer to these as the known cases.

We stumbled upon the empirical fact that $\sF_{3,23,x^{23},\chi_2}\otimes \alpha^{-\deg}$ seemed to have finite (arithmetic and geometric) monodromy, although it was not one of the known cases. As we will prove below, the monodromy is in fact finite.
A computer search for each of $p=3,5,7,11$ and each $2 \le D \le 10^6$ found no other cases than this one and the known cases with finite monodromy. It is not clear whether there should be infinitely many $(p,D)$ other that the known ones with finite monodromy, or finitely many,
or just this one.

In this section, we prove that $\sF_{3,23,x^{23},\chi_2}\otimes \alpha^{-\deg}$ has finite arithmetic and geometric monodromy groups. More generally, we prove that the two-parameter family 
$$\sG(3, 23, 1,5,\chi_2),$$
whose traces at points $(s,t) \in \A^2(K)$, for $K/F_3$ a finite extension, are the sums
$$(s,t) \mapsto -\sum_{x \in K^\times} \psi_K(x^{23} +sx^5 + tx)\chi_{2,K}(x)/\alpha^{-\deg(K/\F_3)},$$
has finite  arithmetic and geometric monodromy groups. We will do so by applying the criterion from Theorem \ref{Vmellin3}

\begin{thm}
 The two-parameter family $$\sG(3, 23, 1,5,\chi_2)$$ has finite arithmetic monodromy.
\end{thm}

\begin{proof}
 We will prove that for every positive integer $f$ and every pair of integers $0\leq x,y<3^f$ we have the inequality
 $$
 \left[23x+5y+\frac{3^f-1}{2}\right]_3\leq [x]_3+[y]_3+f+2.
 $$
 The result follows then from Theorem \ref{Vmellin3}.
 
We proceed by induction on $f$: for $f\leq 4$ one checks it by hand. Let $f\geq 5$ and $0\leq x,y<3^f$. 
 
\smallskip
{\it Case 1}: $x\equiv 0$ (mod 3). 
 
Write $x=3a, y=3c+d$ with 
$$a,c<3^{f-1},~d=0,1,2.$$ 
Then $[5d+1]_3 \leq[d]_3+1$ (check by hand), so 
$$\begin{aligned}
 \left[23x+5y+\frac{3^f-1}{2}\right]_3 & =\left[3\left(23a+5c+\frac{3^{f-1}-1}{2}\right)+5d+1\right]_3\\
& \leq\left[23a+5c+\frac{3^{f-1}-1}{2}\right]_3+[5d+1]_3\\
& \leq[a]_3+[c]_3+(f+1)+[d]_3+1\\ &=[x]_3+[y]_3+f+2\end{aligned}
 $$
by induction.

\medskip
{\it Case 2}: $x\equiv 1$ (mod 3).

Write $x=9a+b$, $y=9c+d$ with 
$$a,c<3^{f-2},~b\in\{1,4,7\},~d<9.$$ 
Then $[23b+5d+4]_3\leq [b]_3+[d]_3+2$ (check by hand), so
$$\begin{aligned}
 \left[23x+5y+\frac{3^f-1}{2}\right]_3 & =\left[9\left(23a+5c+\frac{3^{f-2}-1}{2}\right)+23b+5d+4\right]_3\\
& \leq\left[23a+5c+\frac{3^{f-2}-1}{2}\right]_3+[23b+5d+4]_3\\
& \leq[a]_3+[c]_3+f+[b]_3+[d]_3+2\\ & =[x]_3+[y]_3+f+2\end{aligned}
 $$
 by induction.

\smallskip
{\it Case 3}:  $x\equiv 2$ (mod 3) but $x\not\equiv 8,17$ or $20$ (mod 27).
 
Write $x=27a+b$, $y=27c+d$ with 
$$a,c<3^{f-3},~b\in\{2, 5, 11, 14, 23, 26\},~d<27.$$ Then $[23b+5d+13]_3\leq [b]_3+[d]_3+3$ (check by hand), so
$$\begin{aligned}
 \left[23x+5y+\frac{3^f-1}{2}\right]_3 & =\left[27\left(23a+5c+\frac{3^{f-3}-1}{2}\right)+23b+5d+13\right]_3\\
& \leq\left[23a+5c+\frac{3^{f-3}-1}{2}\right]_3+[23b+5d+13]_3\\
& \leq[a]_3+[c]_3+(f-1)+[b]_3+[d]_3+3\\ & =[x]_3+[y]_3+f+2\end{aligned}
$$
by induction.

\smallskip
{\it Case 4}: $x\equiv 8, 17$ or $20$ (mod 27).
 
Write $x=81a+b$, $y=81c+d$ with 
$$a,c<3^{f-4},~b\in\{8,17,20,35,44,47,62,71,74\},~d<81.$$ 
Then $[23b+5d+40]_3\leq [b]_3+[d]_3+4$ (check by hand), so
$$\begin{aligned}
 \left[23x+5y+\frac{3^f-1}{2}\right]_3 & =\left[81\left(23a+5c+\frac{3^{f-4}-1}{2}\right)+23b+5d+40\right]_3\\
& \leq\left[23a+5c+\frac{3^{f-4}-1}{2}\right]_3+[23b+5d+40]_3\\
& \leq[a]_3+[c]_3+(f-2)+[b]_3+[d]_3+4\\ & =[x]_3+[y]_3+f+2\end{aligned}
$$
by induction.
\end{proof}

\section{Determination of the monodromy groups}
In this section, we will show that with the correct choice \cite[1.7]{Ka-NG2} of $\alpha$, namely 
$$\alpha := -g(\psi,\chi_2),$$
with $g(\psi,\chi_2)$ the quadratic Gauss sum over $\F_3$, we can determine the monodromy of $\sF_{3,23,x^{23},\chi_2}\otimes \alpha^{-\deg}$, and of some other related local systems, exactly. Recall from Lemma \ref{autoduality} that we have
$$G_{geom} \lhd G_{arith} < \SO_{23}(\overline{\Q_\ell}).$$
Thus $G_{geom}$ is an irreducible, primitive (by Lemma \ref{primitivity}) finite subgroup of $\SO_{23}(\overline{\Q_\ell})$. The larger group $G_{arith}$ is a fortiori also an irreducible, primitive finite  subgroup of $\SO_{23}(\overline{\Q_\ell})$.

The traces attained by $\sF_{3,23,x^{23},\chi_2}\otimes \alpha^{-\deg}$, i.e, the traces of elements of $G_{arith}$, are all integers (being algebraic integers in $\Q(\zeta_3)^+=\Q$). Over the field $\F_{81}$, the traces obtained are, by direct calculation, $\{-2,-1,0,1,2,3\}$. Over $\F_{243}$, the traces attained are, by direct calculation, $\{-5,-2,-1,0,1,2\}$. 

Finally, we recall that from Lemma \ref{p-group} that the image of the wild inertia group is the additive group of $\F_{3^5}$, the least extension of $\F_3$ containing the $22$'nd roots of unity.

First we prove the following theorem on finite subgroups of $\SL_{23}(\CC)$:

\begin{thm}\label{simple}
Let $V = \CC^{23}$ and let $G < \SL(V)$ be a finite irreducible subgroup. Let $\chi$ denote the character of $G$ 
afforded by $V$, and suppose that all the following conditions hold:
\begin{enumerate}[\rm(i)]
\item $\chi$ is real-valued;
\item $\chi$ is primitive;
\item $\chi(g) < -1$ for some $g \in G$;
\item The $3$-rank of $G$ is at least $5$.
\end{enumerate}
Then $G \cong \Co_3$ in its unique (orthogonal) irreducible representation of degree $23$.
\end{thm}

\begin{proof}
By the assumption, the $G$-module $V$ is irreducible and primitive; furthermore, it is tensor indecomposable and not tensor induced
since $\dim V = 23$ is prime. Next, we observe by Schur's Lemma that condition (i) implies $\bfZ(G) = 1$. Now we can apply 
\cite[Proposition 2.8]{G-T} (noting that the subgroup $H$ in its proof is just $G$ since $G < \SL(V)$) and arrive at one of the following
two cases.

\smallskip
(a) {\it Extraspecial case}: $P \lhd G$ for some extraspecial $23$-group of order $23^3$ that acts irreducibly on $V$. But in this case,
$\chi|_P$ cannot be real-valued (in fact, $\Q(\chi|_P)$ would be $\Q(\exp(2\pi i/23))$, violating (i).

\smallskip
(b) {\it Almost simple case}: $S \lhd G \leq \Aut(S)$ for some finite non-abelian simple group $S$. In this case, we can apply the main
result of \cite{H-M} and arrive at one of the following possibilities for $S$.

$\bullet$ $S = \Alt_{24}$, $M_{24}$, or $\PSL_2(23)$. Correspondingly, we have that
$G = \Alt_{24}$ or $\SSS_{24}$, $M_{24}$, and $\PSL_2(23)$ or $\PGL_2(23)$. In all of these possibilities, 
$\chi(x) \geq -1$ for $x \in G$ by \cite{ATLAS}, violating (iii).

$\bullet$ $S = \PSL_2(47)$. This is ruled out since $\Q(\chi|_S)$ would be $\Q(\sqrt{-47})$, violating (i).    

$\bullet$ $S = \Co_2$. In this case, $G = S$, and a Sylow $3$-subgroup $P$ of $G$ has a normal extraspecial 
$3$-subgroup $Q \cong 3^{1+4}_+$ of index $3$, see \cite{ATLAS}. Since $G$ has $3$-rank $\geq 5$ by (iv), $P$ contains 
an elementary abelian $3$-subgroup of order $3^5$, whence $Q$ contains a subgroup $R \cong C_3^4$. Identifying $\bfZ(Q)$ with 
$\F_3$, we see that the commutator map induces a non-degenerate symplectic form on $Q/\bfZ(Q) \cong \F_3^4$. As $R$ is abelian,
this form is totally isotropic on $R\bfZ(Q)/\bfZ(Q)$ which has order at least $3^3$. But this is a contradiction, since any isotropic subspace
in $\F_3^4$ has dimension at most $2$.

$\bullet$ $S = \Co_3$. In this case $G = \Co_3$, as stated. 
\end{proof}

\begin{thm}\label{exactmono}We have the following results.
\begin{enumerate}[\rm(i)]
\item For $\sF_{3,23,x^{23},\chi_2}\otimes \alpha^{-\deg}$, we have $G_{geom}=G_{arith}=\Co_3$, the Conway group $\Co_3$, in its
irreducible orthogonal representation of dimension $23$.
\item For the two-parameter family $\sG(3, 23, 1,5,\chi_2)$, we have $G_{geom}=G_{arith}=\Co_3$.
\item The local system on $\G_m \times \A^2/\F_3$ whose trace is given as follows: for any finite extension $K/\F_3$, and any 
$(r \neq 0, s,t)\in \G_m(K) \times \A^2(K)$,
$$(r,s,t)\mapsto -\chi_{2,K}(r)\sum_{x \in K^\times}\psi_K(rx^{23} + sx^5 +tx)\chi_{2,K}(x)/\alpha^{-\deg(K/\F_p)},$$
has $G_{geom}=G_{arith}=\Co_3$.
\item For any finite extension $K/\F_3$, and any $s_0 \in K$, the local system on $\A^1/K$ whose trace function, at points $t \in L$, $L$ a finite extension of $K$, is given by
$$t \mapsto -\sum_{x \in L^\times}\psi_L(x^{23} + s_0x^5 +tx)\chi_{2,L}(x)/\alpha^{-\deg(L/\F_p)},$$
has $G_{geom}=G_{arith}=\Co_3$.
\item The local system on $\A^1/\F_3$ whose trace function, at points $t \in K$, $K$ a finite extension of $\F_3$, is given by
$$t \mapsto -\sum_{x \in K^\times}\psi_L(x^{23} + tx^5)\chi_{2,K}(x)/\alpha^{-\deg(K/\F_p)},$$
has $G_{geom}=G_{arith}=\Co_3$. 
\item For any finite extension $K$ of $\F_3$, and any $(s_0,t_0) \in \A^2(K)$ other than $(0,0)$, the local system on $\G_m/K$ whose trace function at points $r \in L^\times$, $L$ a finite extension of $K$, is given by
$$r \mapsto -\chi_{2,L}(r)\sum_{x \in L^\times}\psi_L(rx^{23} + s_0x^5 +t_0x)\chi_{2,L}(x)/\alpha^{-\deg(L/\F_p)},$$
has $G_{geom}=G_{arith}=\Co_3$.
\end{enumerate}
\end{thm}

\begin{proof}
We first note that among finite irreducible subgroups of $\SO_{23}(\CC)$, the Conway group $\Co_3$ is both maximal {\bf and} minimal.
Indeed, the minimality is clear from the \cite{ATLAS} list of maximal subgroups of $\Co_3$ and their character tables. The maximality is clear
from Theorem \ref{simple}.

The local system in (iv) has finite monodromy because its restriction to the dense open set $(\G_m)^3/\F_3$ has finite monodromy, being the $\sG_{big}$ partner, in the sense of Lemma \ref{big} of two-parameter local system of (ii). The $\chi_2(t)$ term in front keeps its $G_{arith}$ in $SO(23)$, by  \cite[1.7]{Ka-NG2}. It is geometrically
irreducible because this is so already after pullback to the $(1,0,t)$ $t$-line, where it is $\sF_{3,23,x^{23},\chi_2}\otimes \alpha^{-\deg}$.

We remark that the local system in (v) is geometrically irreducible, because it is the Fourier Transform of 
$[5]_\star (\sL_{\psi(x^{23})}\otimes \sL_{\chi_2(x)})$, a middle extension sheaf (cf. \cite[proof of 3.3.1]{Ka-TLFM}) which is geometrically irreducible. Indeed it is   $I(\infty)$-irreducible, because its five $I(\infty)$-slopes are each $23/5$, with exact denominator $5$.

We also remark that the local system in (vi) is geometrically irreducible, because it is $\sL_{\chi_2(t)}$ tensored with the Fourier Transform of 
$[23]_\star (\sL_{\psi(s_0x^{5}+r_0x)}\otimes \sL_{\chi_2(x)})$, a middle extension sheaf (cf. \cite[proof of 3.3.1]{Ka-TLFM}) which is geometrically irreducible. Indeed it is   $I(\infty)$-irreducible, because its $23$ $I(\infty)$-slopes are each either $5/23$, with exact denominator $23$, if $s_0 \neq 0$, or each $1/23$ if $s_0=0$ but $r_0 \neq 0$.

Thus it suffices to prove (i). For then (ii) and (iii) follows from (i) and the maximality of $\Co_3$ as a finite irreducible subgroup of $\SO_{23}(\CC)$, and then (iv), (v) and (vi) each follow from (ii) and (iii) and  the minimality of  $\Co_3$ as a finite irreducible subgroup of $\SO_{23}(\CC)$.

For case (i), we apply Theorem \ref{simple} to $G_{arith}$, to conclude that $G_{arith}=\Co_3$. Then we use minimality of $\Co_3$ to conclude that $G_{geom}=\Co_3$.
\end{proof}


\begin{thebibliography}{99}
\bibitem[ATLAS]{ATLAS} Conway, J. H.,  Curtis, R. T.,  Norton, S. P.,  Parker, R. A. and 
 Wilson, R. A.,  Atlas of finite groups. Maximal subgroups and ordinary characters for simple groups. With computational assistance from J. G. Thackray. Oxford University Press, Eynsham, 1985.
\medskip
 \bibitem[De-Weil II]{De-Weil II}Deligne, P., La conjecture de Weil II.
Publ. Math. IHES 52 (1981), 313-428.
\medskip
\bibitem[G-K-T]{G-K-T}Guralnick, R. M., Katz, N., and Tiep, P. H., Rigid local systems and alternating groups, Tunisian Jour. Math., to appear.
\medskip
\bibitem[G-T]{G-T} Guralnick, R. M. and Tiep, P. H., Symmetric powers and a conjecture of Kollar and Larsen, Invent. Math. 174 (2008), 505--554.
\medskip
\bibitem[H-M]{H-M}
  Hiss, G. and Malle, G., Low-dimensional representations of quasi-simple groups, LMS J. Comput. Math. 4 (2001), 22--63.
\medskip  
\bibitem[Ka-ESDE]{Ka-ESDE}Katz, N., Exponential sums and 
differential equations. Annals of Mathematics Studies, 124. Princeton Univ. Press, Princeton, NJ, 1990. xii+430 pp.
\medskip
\bibitem[Ka-G2hyper]{Ka-G2hyper}Katz, N., $G_2$ and hypergeometric sheaves, Finite Flelds Appl. 13 (2007), no. 2 175-223.
\medskip
\bibitem[Ka-MG]{Ka-MG}Katz, N., On the monodromy groups attached to certain families of exponential sums. Duke Math. J. 54 (1987), no. 1, 41-56. 
\medskip
 \bibitem[Ka-MGcorr]{Ka-MGcorr}Katz, N., Correction to: "On the monodromy groups attached to certain families of exponential sums''. Duke Math. J. 89 (1997), no. 1, 201.
\medskip
 \bibitem[Ka-NG2]{Ka-NG2}Katz, N., Notes on $G_2$, determinants, and equidistribution, Finite Flelds Appl. 10 (2004), 221-269.
\medskip
\bibitem[Ka-RL]{Ka-RL}Katz, N., and Rojas-Le\'{o}n, A., Rigid local systems with monodromy group $2.J_2$, preprint, arXiv:1809.05755.
\medskip
\bibitem[Ka-RLSA]{Ka-RLSA}Katz, N.,  Rigid local systems on $\A^1$ with finite monodromy, Mathematika 64 (2018) 785-846.
\medskip
\bibitem[Ka-Ti-RLSMFUG]{Ka-Ti-RLSMFUG}Katz, N., and Tiep, P.H.,  Rigid local systems,  moments, and finite unitary groups, available on
\url{https://web.math.princeton.edu/~nmk/kt8_22.pdf}.
\medskip
\bibitem[Ka-TLFM]{Ka-TLFM}Twisted L-functions and monodromy. 
Annals of Mathematics Studies, 150. Princeton University Press, Princeton, NJ, 2002. viii+249 pp.
\medskip
\bibitem[Lau]{Lau}Laumon, G., Transformation de Fourier, constantes d'\'{e}quations fonctionnelles et conjecture de Weil, Publ. Math. IHES 65 (1987) 131-210.
\medskip
\bibitem[R-L]{R-L}Rojas-Le\'{o}n, A., Finite monodromy of some families of exponential sums, J. Number. Th. (2018), \url{https://doi.org/10.1016/j.jnt.2018.06.012}
\medskip
\bibitem[Such]{Such} \v{S}uch, Ondrej, Monodromy of Airy and Kloosterman sheaves, Duke Math. J. 103 (2000), no. 3, 397-444.
\medskip
\end{thebibliography}
\end{document}